\documentclass{amsart}
\usepackage{mathptmx, amssymb, enumerate, colonequals, amscd}

\newtheorem{theorem}{Theorem}[section]
\newtheorem{lemma}[theorem]{Lemma}
\newtheorem{proposition}[theorem]{Proposition}

\theoremstyle{definition}
\newtheorem{example}[theorem]{Example}
\newtheorem{remark}[theorem]{Remark}
\newtheorem{question}[theorem]{Question}
\numberwithin{equation}{theorem}

\def\ge{\geqslant}

\def\lf{\lfloor}
\def\rf{\rfloor}

\def\divv{\operatorname{div}}
\def\edim{\operatorname{edim}}
\def\ff{\operatorname{frac}}
\def\rank{\operatorname{rank}}
\def\supp{\operatorname{supp}}
\def\Proj{\operatorname{Proj}}

\def\CC{\mathbb{C}}
\def\FF{\mathbb{F}}
\def\NN{\mathbb{N}}
\def\PP{\mathbb{P}}
\def\QQ{\mathbb{Q}}
\def\ZZ{\mathbb{Z}}
\def\calO{\mathcal{O}}

\begin{document}
\title{Homogeneous prime elements in normal two-dimensional graded rings}

\author{Anurag K. Singh}
\address{Department of Mathematics, University of Utah, 155 South 1400 East, Salt Lake City, UT~84112, USA}
\email{singh@math.utah.edu}

\author{Ryo Takahashi}
\address{Graduate School of Mathematics, Nagoya University, Furocho, Chikusaku, Nagoya, 464-8602, Japan}
\email{takahashi@math.nagoya-u.ac.jp}

\author{Kei-ichi Watanabe}
\address{Department of Mathematics, College of Humanities and Sciences, Nihon University, Setagaya-ku, Tokyo, 156-8550, Japan}
\email{watanabe@math.chs.nihon-u.ac.jp}

\thanks{A.K.S.~was supported by NSF grants DMS~1500613 and DMS~1801285, R.T.~by JSPS Grant-in-Aid for Scientific Research 16H03923 and 16K05098, and K.W.~by Grant-in-Aid for Scientific Research 26400053. This paper started from conversations between R.T.~and K.W.~at the program \emph{Commutative algebra and related topics} at Okinawa, Japan, supported by MSRI, OIST, and RIMS, Kyoto University. The authors are grateful to these institutes, and to the organizers of the program for providing a very comfortable research atmosphere.}

\dedicatory{Dedicated to Professor Craig Huneke on the occasion of his sixty-fifth birthday}

\begin{abstract}
We prove necessary and sufficient conditions for the existence of homogeneous prime elements in normal $\NN$-graded rings of dimension two, in terms of rational coefficient Weil divisors on projective curves.
\end{abstract}
\maketitle

\section{Introduction}

We investigate the existence of homogeneous prime elements, equivalently, of homogeneous principal prime ideals, in normal $\NN$-graded rings $R$ of dimension two. It turns out that there are elegant necessary as well as sufficient conditions for the existence of such prime ideals in terms of rational coefficient Weil divisors, i.e., $\QQ$-divisors, on $\Proj R$.

When speaking of an \emph{$\NN$-graded ring $R$}, we assume throughout this paper that $R$ is a finitely generated algebra over its subring $R_0$, and that $R_0$ is an algebraically closed field. We say that an~$\NN$-grading on $R$ is \emph{irredundant} if
\[
\gcd\{n\in\NN\mid R_n \ne 0\}\ =\ 1.
\]
Relevant material on $\QQ$-divisors is summarized in \S\ref{sec:q:divisors}. Our main result is: 

\begin{theorem}
\label{theorem:main}
Let $R$ be a normal ring of dimension $2$, with an irredundant $\NN$-grading, where $R_0$ is an algebraically closed field. Set $X\colonequals\Proj R$, and let~$D$ be a $\QQ$-divisor on $X$ such that $R=\oplus_{n\ge0} H^0(X,\calO_X(nD))T^n$. Let $d$ be a positive integer.

\begin{enumerate}[\ \rm(1)]
\item Suppose $x\in R_d$ is a prime element. Set
\[
s\colonequals\gcd\{n\in\NN\mid {[R/xR]}_n \ne 0\}.
\]
Then the integers $d$ and $s$ are relatively prime, and the divisor $sdD$ is linearly equivalent to a point of $X$. In particular, $\deg D=1/sd$.

\item Conversely, suppose $dD$ is linearly equivalent to a point $P$ with $P\notin\supp(\ff(D))$. Let $g$ be a rational function on $X$ with
\[
\divv(g)\ =\ P-dD.
\]
Then $x\colonequals gT^d$ is a prime element, and the induced grading on~$R/xR$ is irredundant.
\end{enumerate}
\end{theorem}

The proof of the theorem and further results regarding the number of homogeneous principal prime ideals are included in \S\ref{sec:main}. We next record various examples.

\begin{example}
\label{example:standard}
If a standard $\NN$-graded ring $R$, as in the theorem, has a homogeneous prime element, we claim that $R$ must be a polynomial ring over $R_0$.

If $x\in R_d$ is a prime element, the theorem implies that $d=1$. Independent of the theorem, note that~$R/xR$ is an $\NN$-graded domain of dimension $1$, with ${[R/xR]}_0$ algebraically closed, so~$R/xR$ is a numerical semigroup ring by \cite[Proposition~2.2.11]{GW}. Since it is standard graded, $R/xR$ must be a polynomial ring. But then $R$ is a polynomial ring as well.
\end{example}

\begin{example}
The hypothesis that the underlying field $R_0$ is algebraically closed is crucial in Theorem~\ref{theorem:main} and Example~\ref{example:standard}: the standard graded ring $\QQ[x,y,z]/(x^2+y^2+z^2)$ has a homogeneous prime element $x$.
\end{example}

\begin{example}
In view of Example~\ref{example:standard}, the ring $R\colonequals\CC[x,y,z]/(x^2-yz)$, with the standard $\NN$-grading, has no homogeneous prime element. However, for nonstandard gradings, there can be homogeneous prime elements:

Fix such a grading with $\deg x=a$, $\deg y=b$, and $\deg z=2a-b$, where $\gcd(a,b)=1$ and $b$ is even. Then one has a homogeneous prime element
\[
y^{a-b/2}-z^{b/2},
\]
which generates the kernel of the $\CC$-algebra homomorphism $R\longrightarrow\CC[t]$ with
\[
x\longmapsto t^a,\quad y\longmapsto t^b,\quad z\longmapsto t^{2a-b}.
\]
\end{example}

\begin{example}
The ring $\CC[x,y,z]/(x^4+y^2z+xz^2)$, with $\deg x=4$, $\deg y=5$, and $\deg z=6$, has no homogeneous prime elements in view of Theorem~\ref{theorem:main}\,(1), since the corresponding $\QQ$-divisor has degree $2/15$ by Proposition~\ref{proposition:Tomari}.
\end{example}

\begin{example}
\label{example:236}
Consider $\CC[x,y,z]/(x^2+y^3+z^6)$, with $\deg x=3$, $\deg y=2$, and $\deg z=1$. Then $(z)$ is the unique homogeneous principal prime ideal: the corresponding $\QQ$-divisor has degree $1$, again by Proposition~\ref{proposition:Tomari}.
\end{example}

\begin{example}
Set $R\colonequals\CC[x,y,z]/(x^2-y^3+z^7)$, with $\deg x=21$, $\deg y=14$, and $\deg z=6$. Then the corresponding $\QQ$-divisor has degree $1/42$ by Proposition~\ref{proposition:Tomari}, so the degree of a homogeneous prime element must divide $42$. In view of the degrees of the generators of~$R$, the possibilities are $6$, $14$, $21$, and $42$, and indeed there are prime elements with each of these degrees, namely $z$, $y$, $x$, and $y^3-\lambda x^2$ for scalars $\lambda\neq 0,1$, see also Example~\ref{example:237}. 
\end{example}

\section{Rational coefficient Weil divisors}
\label{sec:q:divisors}

We review the construction of normal graded rings in terms of $\QQ$-divisors; this is work of Dolga\v cev~\cite{Dolgachev}, Pinkham~\cite{Pinkham}, and Demazure~\cite{Demazure}. Let $X$ be a normal projective variety. A \emph{$\QQ$-divisor} on $X$ is a $\QQ$-linear combination of irreducible subvarieties of $X$ of codimension one. Let $D=\sum n_iV_i$ be such a divisor, where $n_i\in\QQ$, and $V_i$ are distinct. Set
\[
\lf D\rf\colonequals\sum\lf n_i\rf V_i,
\]
where $\lf n\rf$ is the greatest integer less than or equal to $n$. We define
\[
\calO_X(D)\colonequals\calO_X(\lf D\rf).
\]
The divisor $D$ is \emph{effective}, denoted $D\ge 0$, if each $n_i$ is nonnegative. The \emph{support} of the fractional part of $D$ is the set
\[
\supp(\ff(D))\colonequals\{V_i\mid n_i\notin\ZZ\}.
\]

Let $K(X)$ denote the field of rational functions on $X$. Each $g\in K(X)$ defines a Weil divisor~$\divv(g)$ by considering the zeros and poles of $g$ with appropriate multiplicity. As these multiplicities are integers, it follows that for a $\QQ$-divisor $D$ one has
\begin{multline*}
H^0(X,\calO_X(\lf D\rf)) = \{g\in K(X)\mid\divv(g)+\lf D\rf \ge 0 \} \\
=\ \{g\in K(X)\mid\divv(g)+ D\ge 0 \} = H^0(X,\calO_X(D)).
\end{multline*}

A $\QQ$-divisor $D$ is \emph{ample} if $ND$ is an ample Cartier divisor for some $N\in\NN$. In this case, the \emph{generalized section ring} $R(X,D)$ is the $\NN$-graded ring 
\[
R(X,D)\colonequals\oplus_{n\ge0}H^0(X,\calO_X(nD))T^n,
\]
where $T$ is an element of degree $1$, transcendental over $K(X)$.

\begin{theorem}[{\cite[3.5]{Demazure}}]
\label{theorem:Demazure}
Let $R$ be an $\NN$-graded normal domain that is finitely generated over a field $R_0$. Let $T$ be a homogeneous element of degree $1$ in the fraction field of $R$. Then there exists a unique ample $\QQ$-divisor $D$ on $X\colonequals\Proj R$ such that
\[
R_n\ =\ H^0(X,\calO_X(nD))T^n\quad\text{ for each }n\ge 0.
\]
\end{theorem}

The following result is due to Tomari:

\begin{proposition}[{\cite[Proposition~2.1]{Tomari}}]
\label{proposition:Tomari}
For $R$ and $D$ as in the theorem above, one has
\[
\lim_{t\to 1}\,(1-t)^{\dim R} P(R,t)\ =\ (\deg D)^{\dim R-1},
\]
where $P(R,t)$ is the Hilbert series of $R$.
\end{proposition}

\section{Homogeneous prime elements}
\label{sec:main}

Before proceeding with the proof of the main theorem, we record a lemma:

\begin{lemma}
\label{lemma:veronese}
Let $R$ be a domain with an irredundant $\NN$-grading. Let $x$ be a nonzero element of degree $d>0$, and set
\[
s\colonequals\gcd\{n\in\NN\mid {[R/xR]}_n \ne 0\}.
\]
Then $\gcd(d,s)=1$. Moreover, $x$ is a prime element of $R$ if and only if $x^s$ is a prime element of the Veronese subring $R^{(s)}\colonequals\oplus_{n\ge 0}R_{ns}$. 
\end{lemma}

\begin{proof}
Note that $P(R/xR,t)$ is a rational function of $t^s$, and that
\[
P(R,t)\ =\ \frac{1}{1-t^d}P(R/xR,t).
\]
Since the grading on $R$ is irredundant, it follows that $\gcd(d,s)=1$.

We claim that $(xR)^{(s)} = x^sR^{(s)}$. Choose a homogeneous element of $(xR)^{(s)}$, and express it as $rx^m$ with~$m$ largest possible. Suppose $m$ is not a multiple of $s$. By considering its degree, we see that the image of $r$ must be $0$ in $R/xR$, contradicting the maximality of $m$. It follows that $(xR)^{(s)} \subseteq x^sR^{(s)}$, the reverse containment being trivial. Hence
\[
R/xR\ =\ (R/xR)^{(s)}\ =\ R^{(s)}/x^sR^{(s)},
\]
which gives the desired equivalence.
\end{proof}

\begin{proof}[Proof of Theorem~\ref{theorem:main}]
For a prime element $x\in R_d$, the ring~$R/xR$ is an~$\NN$-graded domain of dimension $1$, over an algebraically closed field, so \cite[Proposition~2.2.11]{GW} implies that it is isomorphic to a numerical semigroup ring. Take $s$ as in Lemma~\ref{lemma:veronese}. Since the Veronese subring $R^{(s)}$ corresponds to the $\QQ$-divisor $sD$, the proof of (1) reduces using the lemma to the case where $s=1$. In this case,
\[
P(R/xR,t)\ =\ \frac{1}{1-t} - p(t)
\]
for $p(t)$ a polynomial, so
\[
P(R,t)\ =\ \frac{1}{1-t^d}\left(\frac{1}{1-t} - p(t)\right),
\]
and Proposition~\ref{proposition:Tomari} shows that $\deg D =1/d$. To complete the proof of (1), it remains to verify that $dD$ is linearly equivalent to a point of $X$.

The exact sequence
\[
\CD
0 @>>> R(-d) @>x>> R @>>> R/xR @>>> 0
\endCD
\]
shows that for $n\gg0$ one has
\[
\rank R_{n+d}\ =\ 1+\rank R_n.
\]
Choose $m\gg0$ such that $mdD$ is integral, and the above holds with $n=md$, i.e.,
\[
\rank H^0(X,\calO_X(mdD+dD))\ =\ 1 + \rank H^0(X,\calO_X(mdD)).
\]
Let $g$ be a rational function on $X$ such that $x=gT^d$. Then $\divv(g)+dD\ge 0$, and
\[
\rank H^0(X,\calO_X(mdD+\divv(g)+dD))\ =\ 1 + \rank H^0(X,\calO_X(mdD)).
\]
Since $mdD$ is an integral divisor, it follows that
\[
\lf \divv(g)+dD\rf\ \neq\ 0.
\]
Bearing in mind that $\divv(g)+dD$ is an effective divisor of degree $1$, it follows that
\[
\divv(g)+dD\ =\ P,
\]
for $P$ a point of $X$.

For (2), we claim that
\[
xR\ =\ \oplus_{n\ge 0} H^0(X,\calO_X(nD-P))T^n,
\]
and that this is a prime ideal of $R$. Note that homogeneous elements of $xR$ have the form
\[
gT^dhT^m
\]
for $hT^m\in R$, i.e., with $h$ satisfying
\[
\divv(h)+mD\ \ge\ 0.
\]
Since $\divv(g)=P-dD$, the above condition is equivalent to
\[
\divv(gh)+(m+d)D-P\ \ge\ 0,
\]
i.e., to the condition that $gh\in H^0(X,\calO_X((m+d)D-P))$. This proves the claim.

To verify that the ideal $xR$ is prime, consider $h_iT^{m_i}$ in $R\smallsetminus xR$, for $i=1,2$. Then
\[
\divv(h_i)+m_iD\ \ge\ 0 \quad\text{whereas}\quad \divv(h_i)+m_iD-P\ \not\ge\ 0.
\]
Since $P$ is not in the support of the fractional part of $D$, it follows that
\[
\divv(h_1h_2)+(m_1+m_2)D-P\ \not\ge\ 0,
\]
and hence that $h_1h_2T^{m_1+m_2}\notin xR$. Thus, $xR$ is indeed prime. It remains
to prove that the grading on $R/xR$ is irredundant. Set
\[
s\colonequals\gcd\{n\in\NN\mid {[R/xR]}_n \ne 0\},
\]
in which case
\[
P(R/xR,t)\ =\ \frac{1}{1-t^s} - p(t^s),
\]
where $p(t^s)$ is a polynomial in $t^s$, and
\[
P(R,t)\ =\ \frac{1}{1-t^d}\left(\frac{1}{1-t^s} - p(t^s)\right).
\]
Proposition~\ref{proposition:Tomari} gives the second equality below,
\[
\frac{1}{ds}\ =\ \lim_{t\to 1}\, (1-t)^2 P(R,t)\ =\ \deg D\ =\ \frac{1}{d},
\]
implying that $s=1$.
\end{proof} 

\begin{example}
\label{example:222}
Take $\PP^1\colonequals\Proj\CC[u,v]$, with points parametrized by $u/v$, and set
\[
D\colonequals \frac{1}{2}(0)+\frac{1}{2}(\infty)-\frac{1}{2}(1).
\]
Then $R\colonequals R(\PP^1,D)$ is the $\CC$-algebra generated by 
\[
x\colonequals\frac{u-v}{v}T^2, \quad y\colonequals\frac{u-v}{u}T^2, \quad z\colonequals\frac{(u-v)^2}{uv}T^3,
\]
i.e., $R$ is the hypersurface $\CC[x,y,z]/(z^2-xy(x-y))$, with $\deg x=2=\deg y$, and $\deg z=3$. Note that $\deg D=1/2$, and that $2D$ is an integral divisor. Theorem~\ref{theorem:main}\,(2) shows that
\[
\oplus_{n\ge 0} H^0(X,\calO_X(nD-P))T^n
\]
is a prime ideal for $P\in\PP^1\smallsetminus\{0,\infty,1\}$. Indeed, for $P=[\lambda:1]$ with $\lambda\neq0,1$, the displayed ideal is the prime $(x-\lambda y)R$. These are precisely the homogeneous principal prime ideals of $R$, with the points $0$, $\infty$, and $1$ that belong to $\supp(\ff(D))$ corresponding respectively to the ideals $xR$, $yR$ and $(x-y)R$ that are not prime.
\end{example}

\begin{remark}
\label{remark:P1}
Let $D$ be a $\QQ$-divisor on $\PP^1$ such that $\deg D=1/d$ where $d$ is a positive integer, and $dD$ is integral. Then the ring $R\colonequals R(\PP^1,D)$ has infinitely many distinct homogeneous principal prime ideals: all points of $\PP^1$ are linearly equivalent, so for each point~$P$ there exists a rational function $g$ with
\[
\divv(g)\ =\ P-dD,
\]
and Theorem~\ref{theorem:main}\,(2) implies that $gT^dR$ is a prime ideal for each point $P$ with
\[
P\ \in\ \PP^1\setminus\supp(\ff(D)).
\]
This explains the infinitely many prime ideals in Example~\ref{example:222}, and also in Example~\ref{example:237} below; the latter, moreover, has homogeneous prime elements of different degrees:
\end{remark}

\begin{example}
\label{example:237}
On $\PP^1\colonequals\Proj\CC[u,v]$, consider the $\QQ$-divisor
\[
D\colonequals \frac{1}{2}(\infty)-\frac{1}{3}(0)-\frac{1}{7}(1).
\]
Then $R\colonequals R(\PP^1,D)$ is the ring $\CC[x,y,z]/(x^2-y^3+z^7)$, where
\[
z\colonequals\frac{u^2(u-v)}{v^3}T^6,
\quad y\colonequals\frac{u^5(u-v)^2}{v^7}T^{14}, 
\quad x\colonequals\frac{u^7(u-v)^3}{v^{10}}T^{21}.
\]
For each point $P=[\lambda:1]$ in $\PP^1\smallsetminus\{0,\infty,1\}$, i.e., with $\lambda\neq0,1$, one has a prime ideal
\[
\oplus_{n\ge 0} H^0(X,\calO_X(nD-P))T^n\ =\ (y^3-\lambda x^2)R.
\]
These, along with $xR$, $yR$, and $zR$, are precisely the homogeneous principal prime ideals.
\end{example}

\begin{example}
\label{example:ec1}
Set $X$ to be the elliptic curve $\Proj\CC[u,v,w]/(v^2w-u^3+w^3)$. Then
\[
\divv(v/w)\ =\ P_1+P_2+P_3-3O,
\]
where $O=[0:1:0]$ is the point at infinity, and 
\[
P_1=[1:0:1], \quad P_2=[\theta:0:1], \quad P_3=[\theta^2:0:1],
\]
for $\theta$ a primitive cube root of unity. Take
\[
D\colonequals\frac{1}{2}P_1+\frac{1}{2}P_2+\frac{1}{2}P_3-O.
\]
The ring $R\colonequals R(X,D)$ has generators
\[
x\colonequals\frac{w}{v}T^2, \quad y\colonequals\frac{w}{v}T^3, \quad z\colonequals\frac{uw}{v^2}T^4,
\]
so $R=\CC[x,y,z]/(x^6+y^4-z^3)$. Since $\deg D=1/2$, the only possible homogeneous prime elements are in degree $2$. Indeed,
\[
2D\ =\ P_1+P_2+P_3-2O\ =\ \divv(v/w)+O,
\]
and $O\notin\supp(\ff(D))$, so $(w/v)T^2=x$ is a prime element; note that $xR$ is the \emph{unique} homogeneous principal prime ideal of $R$, in contrast with Examples~\ref{example:222} and~\ref{example:237}.
\end{example}

\begin{example}
\label{example:ec2}
With $X$ and $O$ as in the previous example, note that
\[
\divv(u/w)\ =\ Q_1+Q_2-2O,
\]
where
\[
Q_1=[0:i:1], \quad Q_2=[0:-i:1].
\]
Consider the $\QQ$-divisor
\[
D\ =\ \frac{1}{2}Q_1+\frac{1}{2}Q_2-\frac{1}{2}O.
\]
Then the ring $R\colonequals R(X,D)$ has generators
\[
x\colonequals\frac{w}{u}T^2, \quad y\colonequals\frac{w}{u}T^3, \quad z\colonequals\frac{w}{u}T^4, \quad t\colonequals\frac{vw^2}{u^3}T^6,
\]
and presentation
\[ 
R\ =\ \CC[x,y,z,t]/(y^2-xz,\ x^6-z^3+t^2).
\]
Once again, since $\deg D=1/2$, the only possibility for homogeneous prime elements is in degree $2$. We see that
\[
2D\ =\ Q_1+Q_2-O\ =\ \divv(u/w)+O.
\]
However, since $O\in\supp(\ff(D))$, Theorem~\ref{theorem:main}\,(2) does not apply. Indeed, $(w/u)T^2=x$ is not a prime element. The key point is that $X$ is not rational, and there does not exist a point $P$, linearly equivalent to $2D$, with $P\notin\supp(\ff(D))$.
\end{example}

\section{Rational singularities}

Let $H$ be a numerical semigroup. For $\FF$ a field and $t$ an indeterminate, set
\[
\FF[H]\colonequals\FF[t^n\mid n\in H].
\]

\begin{question}
Does $\FF[H]$ deform to a normal $\NN$-graded ring, i.e., does there exist a normal $\NN$-graded ring $R$, with $x\in R$ homogeneous, such that $R/xR\cong \FF[H]$?
\end{question}

\begin{question}
For which numerical semigroups $H$ does there exist $R$, as above, such that~$R$ has rational singularities?
\end{question}

The following is a partial answer:

\begin{proposition}
\label{proposition:rat:sing}
Let $R$ be a normal ring of dimension $2$, with an irredundant $\NN$-grading, where $R_0=\FF$ is an algebraically closed field. Suppose $x_0$ is a homogeneous prime element that is part of a minimal reduction of $R_+$, and that the induced grading on $R/x_0R$ is irredundant. Then the following are equivalent:
\begin{enumerate}[\ \rm(1)]
\item The ring $R$ has rational singularities.

\item There exist minimal $\FF$-algebra generators $x_0,\dots,x_r$ for $R$, with $x_i$ homogeneous, and
\begin{equation}
\label{equation:rat:sing}
r+\deg x_0 > \deg x_1 > \cdots > \deg x_r=r.
\end{equation}
\end{enumerate}
\end{proposition}

\begin{proof}
Note that $R/x_0R$ is a numerical semigroup ring; let $H$ denote the semigroup.

(1) $\implies$ (2): The element $x_0$ extends to a minimal generating set $x_0,\dots,x_r$ for~$R$. Since~$R/x_0R=\FF[H]$ is a numerical semigroup ring, the degrees of $x_1,\dots,x_r$ are distinct; after reindexing, we may assume that
\[
\deg x_1 > \cdots > \deg x_r.
\]
Since $R$ is a $2$-dimensional ring with rational singularities, it has minimal multiplicity by~\cite[Theorem~3.1]{HW}, namely
\[
e(R)\ =\ \edim(R)-1.
\]
As $x_0$ is part of a minimal reduction of $R_+$, the ring $R/x_0R$ has minimal multiplicity as well, i.e., $e(R/x_0R)=r$. It follows that $\deg x_r=r$. By \cite[Corollary~3.2]{RG}, the Frobenius number of $H$ is $\deg x_1-\deg x_r = \deg x_1-r$, which is the $a$-invariant of $\FF[H]$. But then
\[
a(R)+\deg x_0\ =\ a(\FF[H])\ =\ \deg x_1-r.
\]
Since $R$ is a ring of positive dimension with rational singularities, $a(R)$ must be negative by \cite{Flenner, Watanabe:k:star}, implying that $r+\deg x_0 > \deg x_1$ as desired.

(2) $\implies$ (1): Since $R$ is normal by assumption, one has only to verify that $a(R)<0$ in view of the above references. This is immediate since the $a$-invariant of $\FF[H]$, equivalently the Frobenius number of $H$, is $\deg x_1-r$.
\end{proof}

\begin{example}
Consider the $\QQ$-divisor
\[
D\colonequals\frac{5}{7}(0)-\frac{4}{7}(\infty)
\]
on $\PP^1\colonequals\Proj\CC[u,v]$, with points parametrized by $u/v$. Then $R\colonequals R(\PP^1,D)$ has generators
\[
w\colonequals\frac{v^2}{u^2}T^3, \quad x\colonequals\frac{v^3}{u^3}T^5, \quad
y\colonequals\frac{v^4}{u^4}T^7, \quad z\colonequals\frac{v^5}{u^5}T^7.
\]
The relations are readily seen to be the size two minors of the matrix
\[
\begin{pmatrix}
w & x & z\\
x & y & w^3
\end{pmatrix}.
\]
Each point $P=[\lambda:1]$ with $\lambda\neq0$ gives a prime ideal
\[
\oplus_{n\ge 0} H^0(X,\calO_X(nD-P))T^n\ =\ (y-\lambda z)R,
\]
and these are precisely the homogeneous principal prime ideals of $R$.

For example,
\[
R/(y-z)R\ =\ \CC[t^3, t^5, t^7].
\]
Since $(y-z, w)R$ is a minimal reduction of $R_+$ and the grading on $R/(y-z)R$ is irredundant, Proposition~\ref{proposition:rat:sing} applies. The ring $R$ has rational singularities since $a(R)=-3$, and the inequalities~\eqref{equation:rat:sing} indeed hold since
\[
3+\deg (y-z) > \deg y > \deg x > \deg w = 3.
\]
\end{example}

\begin{example}
Take $R$ as in Example~\ref{example:236}, i.e., $R\colonequals\CC[x,y,z]/(x^2+y^3+z^6)$, with $\deg x=3$, $\deg y=2$, and $\deg z=1$. Then $z$ is a prime element such that the induced grading on $R/zR$ is irredundant; $z$ is also part of the minimal reduction $(z,y)R$ of $R_+$. Since $a(R)=0$, the ring~$R$ does not have rational singularities; likewise,~\eqref{equation:rat:sing} does not hold since
\[
2+\deg z \not> \deg x.
\]
\end{example}


\end{document}